\documentclass[11pt]{article}
\usepackage{amsmath}
\usepackage{amssymb}
\usepackage{amsthm}
\usepackage{centernot}
\usepackage[usenames]{color}
\usepackage{amscd}
\usepackage{dsfont}
\usepackage{indentfirst}

\usepackage[colorlinks=true,linkcolor=blue,filecolor=red,
citecolor=webgreen]{hyperref}
\definecolor{webgreen}{rgb}{0,.5,0}

\numberwithin{equation}{section}

\def\C{{\mathds{C}}}

\def\N{{\mathds{N}}}

\def\1{{\bf 1}}

\def\lcm{\operatorname{lcm}}

\newcommand{\divides}{\mathrel{\vert}}
\newcommand{\notdivides}{\negthinspace\mathrel{\not\vert}} 

\newtheorem{theorem}{Theorem}[section]

\newtheorem{lemma}[theorem]{Lemma}

\begin{document}

\title{\bf On the number of $k$-compositions of $n$ satisfying certain coprimality conditions}
\author{L\'aszl\'o T\'oth  \\
Department of Mathematics \\
University of P\'ecs \\
Ifj\'us\'ag \'utja 6, 7624 P\'ecs, Hungary \\
E-mail: {\tt ltoth@gamma.ttk.pte.hu}}
\date{}
\maketitle

\centerline{Acta Math. Hungar. {\bf 164} (1) (2021), 135--156}

\begin{abstract} We generalize the asymptotic estimates by Bubboloni, Luca and Spiga (2012) on the number of $k$-compositions of $n$
satisfying some coprimality conditions. We substantially refine the error term concerning the number of $k$-compositions of $n$ with pairwise
relatively prime summands. We use a different approach, based on properties of multiplicative arithmetic functions of $k$ variables and on an
asymptotic formula for the restricted partition function.
\end{abstract}



\medskip
{The research was financed by NKFIH in Hungary, within the framework of the 2020-4.1.1-TKP2020 3rd thematic programme of the University of P\'ecs.}

\section{Motivation}

Let $n,k\in \N:=\{1,2,\ldots \}$. For $n\ge k$, the $k$-compositions of $n$ are the representations $x_1+\cdots+x_k=n$, where
$(x_1,\ldots,x_k)\in \N^k$. It is well known that the number of all $k$-compositions of $n$ is $\binom{n-1}{k-1}$. Gould \cite{Gou1964} considered the
number $R_k(n)$ of $k$-compositions of $n$ such that $\gcd(x_1,\ldots,x_k)=1$, obtained that its Lambert series is
\begin{equation*}
\sum_{n=1}^{\infty} R_k(n) \frac{x^n}{1-x^n} = \frac{x^k}{(1-x)^k},
\end{equation*}
and then deduced that
\begin{equation} \label{R_n_k}
R_k(n) = \sum_{d\divides n} \binom{d-1}{k-1}\mu(n/d),
\end{equation}
where $\mu$ is the M\"{o}bius function. A short direct proof of identity \eqref{R_n_k} is the following.
By using that $\sum_{d\mid n} \mu(d)=1$ or $0$, according to $n=1$ or $n>1$, we obtain that
\begin{equation} \label{proof_R_n_k}
R_k(n) = \sum_{\substack{x_1,\ldots,x_k\ge 1\\x_1+\cdots+ x_k=n}} \sum_{d\divides \gcd(x_1,\ldots,x_k)} \mu(d)
= \sum_{d\divides n} \mu(d) \sum_{\substack{y_1,\ldots,y_k\ge 1\\y_1+\cdots+ y_k=n/d}} 1
= \sum_{d\divides n} \mu(d) \binom{n/d-1}{k-1},
\end{equation}
giving \eqref{R_n_k}.

Let $J_t(n)=n^t \prod_{p\divides n} (1-1/p^t)$ denote the Jordan function of order $t$, where $J_1(n)=\varphi(n)$ is Euler's totient function.
If $k=2$, then it follows that $R_2(n)=\varphi(n)$ for every $n\ge 2$. Identity \eqref{R_n_k} shows that for $k\ge 3$,
\begin{equation} \label{R_k_n_asympt}
R_k(n) = \frac1{(k-1)!}J_{k-1}(n) + O(n^{k-2}),
\end{equation}
as $n\to \infty$. In fact, \eqref{R_n_k} implies that for every $k$, $R_k(n)$ is a linear combination of the Jordan functions $J_t(n)$ ($1\le t\le k-1$).
More exactly,
for every $n\ge k\ge  2$,
\begin{equation*}
R_k(n) = \frac1{(k-1)!} \sum_{t=1}^{k-1} a_{k,t} J_t(n) = \frac1{(k-1)!}\left(J_{k-1}(n) -\frac{k(k-1)}{2} J_{k-2}(n)+\cdots +a_{k,1} \varphi(n) \right),
\end{equation*}
where
\begin{equation*}
a_{k,t}(n) = \sum_{j=t}^{k-1} (-1)^{j-t} s(k-1,j) \binom{j}{t},
\end{equation*}
$s(n,k)$ denoting the signed Stirling numbers of the first kind.

Bubboloni, Luca and Spiga \cite{BLS2012} investigated the number of $k$-compositions of $n$ with some different coprimality constraints.
Namely, let
\begin{equation*}
A_k(n):= \# \{(x_1,\ldots,x_k)\in \N^k: x_1+\cdots +x_k=n, \gcd(x_1,x_2\cdots x_k)=1\}
\end{equation*}
and
\begin{equation*}
B_k(n):= \# \{(x_1,\ldots,x_k)\in \N^k: x_1+\cdots +x_k=n, \gcd(x_i,x_j)=1, 1\le i< j\le k\}.
\end{equation*}

Note that, if $k=2$, then $A_2(n)=B_2(n)=R_2(n)=\varphi(n)$ for every $n\ge 2$. They proved in \cite{BLS2012} that for any $k\ge 3$,
\begin{equation} \label{A_k}
A_k(n) = C_k f_k(n)\frac{n^{k-1}}{(k-1)!} + O\left(n^{k-2} (\log n)^{k-1} \right)
\end{equation}
and
\begin{equation} \label{B_k}
B_k(n) = D_k g_k(n)\frac{n^{k-1}}{(k-1)!} + O\left(n^{k-1} (\log n)^{-1} \right),
\end{equation}
as $n\to \infty$, where $C_k$ and $D_k$ are explicit constants, depending only on $k$, while $f_k(n)$ and $g_k(n)$ are depending on $k$ and on the prime
factorization of $n$. Actually, it was also proved in \cite{BLS2012} that $2/3<f_k(n)<2$ and $1/(2k)<g_k(n)<2k$ for every $n\in \N$,
and explicit constants for the above $O$-terms were given. Note also, that in \cite{BLS2012} the sum $A_k(n)$ was denoted by $A_{k-1}(n)$, and
the corresponding results were formulated for $k+1$ instead of $k$. As described in \cite{BLS2012}, these results have applications in group
theory and Galois theory.

In a recent preprint, Thomas \cite{Tho2020} considered the case $k=3$ and obtained for the sum $B_3(n)$ the error term $O(n^{3/2+\varepsilon})$, which
refines \eqref{B_k}.

In this paper we define for $1\le s \le k-1$,
\begin{equation*}
A_{k,s}(n):= \# \{(x_1,\ldots,x_k)\in \N^k: x_1+\cdots +x_k=n, \gcd(x_1\cdots x_s, x_{s+1}\cdots x_k)=1\},
\end{equation*}
representing the number of $k$-compositions of $n$ such that the first $s$ summands are coprime to the last $k-s$ summands.
We remark the symmetry property $A_{k,s}(n)= A_{k,k-s}(n)$, valid for every $1\le s\le k-1$.
Also, for $2\le t \le k$ we define
\begin{equation*}
B_{k,t}(n):= \# \{(x_1,\ldots,x_k)\in \N^k: x_1+\cdots +x_k=n, \gcd(x_{i_1},\ldots,x_{i_t})=1, 1\le i_1<\ldots < i_t\le k\},
\end{equation*}
representing the number of $k$-compositions of $n$ such that the summands are $t$-wise relatively prime. If $s=1$ and $t=2$, then
$A_{k,1}(n)=A_k(n)$ and $B_{k,2}(n)=B_k(n)$, defined above. Also, if $t=k$, then $A_{k,k}(n)=R_k(n)$.

We obtain asymptotic formulas with error terms for $A_{k,s}(n)$ and $B_{k,t}(n)$, which generalize formulas \eqref{R_k_n_asympt}, \eqref{A_k}
and \eqref{B_k}. In the case $t=2$ our error term substantially refines the error term of formula \eqref{B_k}, obtained
in \cite{BLS2012,Tho2020}. We use a different approach, inspired by \eqref{proof_R_n_k}, based on properties of multiplicative arithmetic
functions of $k$ variables and the restricted partition function
\begin{equation} \label{def_N}
N(n;a_1,\ldots,a_k):= \# \{(x_1,\ldots,x_k)\in \N^k: a_1x_1+\cdots +a_kx_k=n\}.
\end{equation}

The treatments for the sums $A_{k,s}(n)$ and $B_{k,t}(n)$ are quite similar. For the proofs we use the method and some results from our
earlier work \cite{Tot2016}, concerning the asymptotic density of the $k$-tuples of positive integers
with $t$-wise relatively prime components, with some other notation. The main results of the present paper are given in Section
\ref{Section_Main_results}, and their proofs are included in Section \ref{Section_Proofs}. Some preliminaries, needed for the proofs, are presented in Section \ref{Section_Preliminaries}.

We will use the notation $n=\prod_p p^{\nu_p(n)}$ for the prime power factorization of $n\in \N$, the product being over the primes $p$,
where all but a finite number of the exponents $\nu_p(n)$ are zero. The number of distinct prime factors of $n$ will be denoted by $\omega(n)$.

\section{Main results} \label{Section_Main_results}

For $k\ge 3$ and $1\le s\le k-1$ let
\begin{equation} \label{notation_F}
F_{k,s}(x): = \frac1{x} \left(x^k+ (x-1)^k- x^s(x-1)^{k-s}- x^{k-s}(x-1)^s + (-1)^{k-1}\right),
\end{equation}
which is a polynomial of degree $k-3$. Also, for $k\ge 3$ and $2\le t\le k$ let
\begin{align}  \label{notation_G}
G_{k,t}(x): = & \frac1{x} \left(\sum_{j=0}^{t-1} \binom{k}{j}(x-1)^{k-j}+(-1)^{k-t} \binom{k-1}{t-1} \right)
\nonumber
\\ = & \frac1{x} \left(x^k- \sum_{j=t}^{k-1}(-1)^{j-t} \binom{k}{j} \binom{j-1}{t-1} x^{k-j} \right),
\end{align}
which is a polynomial of degree $k-1$.

We prove the following results.

\begin{theorem} \label{Theorem_A} Let $k\ge 3$ and let $1\le s \le k-1$. Then
\begin{equation*}
A_{k,s}(n) = C_{k,s} f_{k,s}(n)\frac{n^{k-1}}{(k-1)!} + \begin{cases} O\left(n^{k-2}(\log n)^{\max(s,k-s)} \right), & \text{if $k\ge 4$,} \\
O\left(n(\log n)^2 \log \log n \right), & \text{if $k=3$,}
\end{cases}
\end{equation*}
where
\begin{equation*}
C_{k,s} = \prod_p \left(1- \frac{F_{k,s}(p)}{p^{k-1}} \right),
\end{equation*}
\begin{equation*}
f_{k,s}(n)= \prod_{p\divides n} \left(1 + \frac{(-1)^{k-1}}{p^{k-1}-F_{k,s}(p)} \right).
\end{equation*}
\end{theorem}

\begin{theorem} \label{Theorem_B} Let $k\ge 3$ and let $2\le t \le k$. Then
\begin{equation*}
B_{k,t}(n) = D_{k,t} g_{k,t}(n) \frac{n^{k-1}}{(k-1)!} + \begin{cases} O\left(n^{k-2} \right), & \text{if $k\ge 4$, $t\ge 3$, or $k=t=3$,} \\
O\left(n^{k-2} (\log n)^{k-1} \right), & \text{if $k\ge 4$, $t=2$,} \\
O\left(n (\log n)^2 (\log \log n)^2 \right), & \text{if $k=3$, $t=2$,}
\end{cases}
\end{equation*}
where
\begin{equation*}
D_{k,t} = \prod_p \frac{G_{k,t}(p)}{p^{k-1}},
\end{equation*}
\begin{equation*}
g_{k,t}(n)= \prod_{p\divides n} \left(1 + \frac{(-1)^{k-t+1}\binom{k-1}{t-1}}{G_{k,t}(p)}\right).
\end{equation*}
\end{theorem}

We note that $f_{k,s}(n)$ and $g_{k,t}(n)$ are positive bounded functions of $n$, which specify the asymptotic behavior of $A_{k,s}(n)$ and
$B_{k,t}(n)$. In the case $s=1$ we recover from Theorem \ref{Theorem_A} formula \eqref{A_k} with the same error term for $k\ge 4$, and with a
slightly weaker error term for $k=3$. In the case $t=2$ we recover from Theorem \ref{Theorem_B} formula \eqref{B_k} with a substantially better error term for every $k\ge 3$.

\section{Preliminaries to the proofs} \label{Section_Preliminaries}

\subsection{Arithmetic functions of several variables}

Let $f:\N^k\to \C$ be an arithmetic function of $k$ variables. Its Dirichlet series is given by
\begin{equation*}
D(f;z_1,\ldots,z_k)= \sum_{n_1,\ldots,n_k=1}^{\infty}
\frac{f(n_1,\ldots,n_k)}{n_1^{z_1}\cdots n_k^{z_k}}.
\end{equation*}

Similar to the one variable case, if $D(f;z_1,\ldots,z_k)$ is
absolutely convergent for $(z_1,\ldots,z_k)\in \C^k$, then it is
absolutely convergent for every $(s_1,\ldots,s_k)\in \C^k$ with $\Re
s_j \ge \Re z_j$ ($1\le j\le k$).

The Dirichlet convolution of the functions $f,g:\N^k\to \C$ is defined by
\begin{equation*}
(f*g)(n_1,\ldots,n_k)= \sum_{d_1\mid n_1, \ldots, d_k\mid n_k}
f(d_1,\ldots,d_k) g(n_1/d_1, \ldots, n_k/d_k).
\end{equation*}

If $D(f;z_1,\ldots,z_k)$ and $D(g;z_1,\ldots,z_k)$ are
absolutely convergent, then
$D(f*g;z_1,\ldots,z_k)$ is also absolutely convergent and
\begin{equation*}
D(f*g;z_1,\ldots,z_k) = D(f;z_1,\ldots,z_k) D(g;z_1,\ldots,z_k).
\end{equation*}

We also recall that a nonzero arithmetic function of $k$ variables $f:\N^k\to \C$ is said to be multiplicative if
\begin{equation*}
f(m_1n_1,\ldots,m_kn_k)= f(m_1,\ldots,m_k) f(n_1,\ldots,n_k)
\end{equation*}
holds for every $m_1,\ldots,m_k,n_1,\ldots,n_k\in \N$ such that $\gcd(m_1\cdots m_k,n_1\cdots n_k)=1$.
If $f$ is multiplicative, then it is determined by the values
$f(p^{\nu_1},\ldots,p^{\nu_k})$, where $p$ is prime and
$\nu_1,\ldots,\nu_k\in \N_0:=\N \cup \{0\}$. More exactly, $f(1,\ldots,1)=1$ and
for every $n_1,\ldots,n_k\in \N$,
\begin{equation*}
f(n_1,\ldots,n_k)= \prod_p f(p^{\nu_p(n_1)}, \ldots,p^{\nu_p(n_k)}),
\end{equation*}
the product being over the primes $p$. If $f,g:\N^k \to \C$ are multiplicative, then their Dirichlet convolution $f*g$ is also multiplicative.
If $k=1$, i.e., in the case of functions of a
single variable we recover the familiar notion of multiplicativity.

If $f$ is multiplicative, then its Dirichlet series can be expanded into a (formal)
Euler product, that is,
\begin{equation} \label{Euler_product}
D(f;z_1,\ldots,z_k)=  \prod_p \sum_{\nu_1,\ldots,\nu_k=0}^{\infty}
\frac{f(p^{\nu_1},\ldots, p^{\nu_k})}{p^{\nu_1z_1+\cdots +\nu_k
z_k}},
\end{equation}
the product being over the primes $p$. More exactly, if $f$ is multiplicative, then the series
$D(f;z_1,\ldots,z_k)$ is absolutely convergent if and only if
\begin{equation*}
\sum_p \sum_{\substack{\nu_1,\ldots,\nu_k=0\\ \nu_1+\cdots +\nu_k \ge 1}}^{\infty}
\frac{|f(p^{\nu_1},\ldots, p^{\nu_k})|}{p^{\nu_1 \Re z_1+\cdots +\nu_k
\Re z_k}} < \infty
\end{equation*}
and in this case equality \eqref{Euler_product} holds. See Delange \cite{Del1969} and the survey by the author \cite{Tot2014}.

\subsection{The restricted partition function}

Let $\N_0:=\{0,1,2,\ldots\}$ and consider the restricted partition function
\begin{equation} \label{def_P}
P(n;a_1,\ldots,a_k):= \# \{(x_1,\ldots,x_k)\in \N_0^k: a_1x_1+\cdots +a_kx_k=n\},
\end{equation}
where $a_1,\ldots,a_k$ are given positive integers. It is known that in the case $\gcd(a_1,\ldots,a_k)=1$ the function $P(n;a_1,\ldots,a_k)$
can be represented as a $(k-1)$-degree polynomial in $n$ plus a periodic sequence of $n$. More exactly, one has
\begin{equation} \label{form_P}
P(n;a_1,\ldots,a_k)= c_{k-1}n^{k-1}+ c_{k-2} n^{k-2}+\cdots + c_1n+c_0 + s(n;a_1,\ldots,a_k),
\end{equation}
where
\begin{equation} \label{coeff_P}
c_{k-1} = \frac1{(k-1)!a_1\cdots a_k}, \quad  c_{k-2}= \frac{a_1+\cdots +a_k}{2(k-2)!a_1\cdots a_k},
\end{equation}
and $s(n;a_1,\ldots,a_k)$ is a periodic sequence with period $\lcm(a_1,\ldots,a_k)$. See Comtet \cite[Sect.\ 2.6, Th.\ C]{Com1974}.
There is a rich bibliography on properties of the restricted partition function, in particular on its polynomial part. See, e.g., the recent
papers \cite{AC2005, CF2018, DV2017} and their references.

For our purposes the variables $x_1,\ldots,x_k$ are positive. If $(x_1,\ldots,x_k)\in \N^k$ is a solution of the equation $a_1x_1+\cdots + a_kx_k=n$, 
then $(y_1,\ldots,y_k)=(x_1-1,\ldots,x_k-1)\in \N_0^k$ is a solution for $a_1y_1+\cdots + a_ky_k=n-a_1-\cdots -a_k$, and conversely. Hence the connection 
between the functions \eqref{def_N} and \eqref{def_P} is given by
\begin{equation} \label{connect_N_P}
N(n;a_1,\ldots,a_k) = P(n-a_1-\cdots-a_k;a_1,\ldots,a_k).
\end{equation}

We need the following result.

\begin{lemma} \label{Lemma_Partition} Let $k\in \N$, $k\ge 2$ be fixed and let $a_1,\ldots,a_k\in \N$ such that $\gcd(a_1,\ldots,a_k)=1$. Then
\begin{equation*}
N(n;a_1,\ldots,a_k) = \frac{n^{k-1}}{(k-1)!a_1\cdots a_k}+ O\left(\frac{n^{k-2}(a_1+\cdots +a_k)}{a_1\cdots a_k} \right),
\end{equation*}
as $n\to \infty$, uniformly for $n,a_1,\ldots, a_k$.
\end{lemma}

\begin{proof} According to \eqref{form_P} and \eqref{coeff_P} we have
\begin{equation*}
P(n;a_1,\ldots,a_k) = \frac{n^{k-1}}{(k-1)!a_1\cdots a_k}+ O\left(\frac{n^{k-2}(a_1+\cdots +a_k)}{a_1\cdots a_k} \right).
\end{equation*}

By using \eqref{connect_N_P} we conclude that the same asymptotics holds for $N(n;a_1,\ldots,a_k)$ as well.
\end{proof}

\section{Proofs} \label{Section_Proofs}

\subsection{The sum $A_{k,s}(n)$}

We define the function
\begin{equation*}
\vartheta_{k,s}(n_1,\ldots,n_k)= \begin{cases} 1, & \text{ if $\gcd(n_1\cdots n_s, n_{s+1} \cdots n_k)=1$}, \\ 0, &
\text{ otherwise,}
\end{cases}
\end{equation*}
which is symmetric in the variables $x_1,\ldots,x_s$ and in $x_{s+1},\ldots,x_k$, but not symmetric in all variables. However, $\vartheta_{k,s}$
is a  multiplicative function of $k$ variables. Hence we have
\begin{equation*}
\vartheta_{k,s}(n_1,\ldots,n_k)= \prod_p \vartheta_{k,s} (p^{\nu_p(n_1)}, \ldots,p^{\nu_p(n_k)})
\end{equation*}
for every $n_1,\ldots,n_k\in \N$. Also, for every prime $p$ and every $\nu_1,\ldots,\nu_k\in \N_0$,
\begin{equation} \label{vartheta_val}
\vartheta_{k,s}(p^{\nu_1}, \ldots,p^{\nu_k})= \begin{cases}
1, & \text{ if $\nu_1=\cdots =\nu_s=0$, $\nu_{s+1},\ldots,\nu_k\in \N_0$}, \\
\ & \text{ or $\nu_1,\ldots,\nu_s\in \N_0$, $\nu_{s+1} =\cdots =\nu_k=0$}, \\
0, & \text{ otherwise.}
\end{cases}
\end{equation}

For the multiple Dirichlet series of the function $\vartheta_{k,s}$ we
have the next result.

\begin{lemma} \label{Lemma_Dir_series} If $1 \le s \le k-1$, then
\begin{equation} \label{product_repr}
\sum_{n_1,\ldots,n_k=1}^{\infty}
\frac{\vartheta_{k,s}(n_1,\ldots,n_k)}{n_1^{z_1}\cdots n_k^{z_k}}=
\zeta(z_1)\cdots \zeta(z_k) H_{k,s}(z_1,\ldots,z_k),
\end{equation}
where $\zeta$ is the Riemann zeta function, and 
\begin{equation} \label{H_k_s}
H_{k,s}(z_1,\ldots,z_k)= \prod_p \left(\prod_{1\le j\le s} \left(1-\frac1{p^{z_j}}\right) + \prod_{s+1\le j\le k} \left(1-\frac1{p^{z_j}} \right)
- \prod_{1\le j\le k} \left(1-\frac1{p^{z_j}}\right) \right),
\end{equation}
absolutely convergent if $\Re z_j \ge 0$ \textup{($1\le j\le k$)} and $\Re (z_j+z_{\ell}) >1$ \textup{($1\le j\le s$, $s+1 \le \ell \le k$)}.
In particular, for $z_1=\cdots =z_k=1$,
\begin{equation*}
H_{k,s}(1,\ldots,1)= \prod_p  \left(1- \frac{p F_{k,s}(p)+(-1)^k}{p^k}\right)
\end{equation*}
is absolutely convergent, where $F_{k,s}(p)$ is defined by \eqref{notation_F}.
\end{lemma}

\begin{proof} The function $\vartheta_{k,s}(n_1,\ldots,n_k)$ is multiplicative, hence its Dirichlet series can be expanded into an
Euler product. Using \eqref{vartheta_val} we deduce
\begin{equation*}
\sum_{n_1,\ldots,n_k=1}^{\infty} \frac{\vartheta_{k,s}(n_1,\ldots,n_k)}{n_1^{z_1}\cdots n_k^{z_k}}=
\prod_p \sum_{\nu_1,\ldots,\nu_k=0}^{\infty}
\frac{\vartheta_{k,s}(p^{\nu_1},\ldots,p^{\nu_k})}{p^{\nu_1 z_1+\cdots +\nu_k z_k}}
\end{equation*}
\begin{equation*}
= \prod_p \left( \sum_{\nu_1,\ldots,\nu_s=0}^{\infty} \frac1{p^{\nu_1 z_1+\cdots +\nu_s z_s}} +
\sum_{\nu_{s+1},\ldots,\nu_k=0}^{\infty} \frac1{p^{\nu_{s+1} z_{s+1} +\cdots +\nu_k z_k}} -1 \right)
\end{equation*}
\begin{equation*}
= \prod_p \left(\prod_{1\le j\le s} \left(1-\frac1{p^{z_j}}\right)^{-1} + \prod_{s+1\le j\le k} \left(1-\frac1{p^{z_j}} \right)^{-1}
- 1 \right),
\end{equation*}
which gives the desired formula by factoring out $\prod_{1\le j\le k} \prod_p \left(1-\frac1{p^{z_j}} \right)^{-1} =\zeta(z_1)\cdots \zeta(z_k)$.
Observe that, performing the multiplications in \eqref{H_k_s}, all terms with $p^e$
cancel out, where $e=z_{i_1}+\cdots +z_{i_j}$ such that $1\le i_1<\cdots <i_j\le s$ or $s+1\le i_1<\cdots <i_j\le k$. This gives the
result on the absolute convergence. For example, if $k=3$ and $s=2$, then
\begin{equation*}
H_{3,2}(z_1,z_2,z_3)=  \prod_p \left( \left(1-\frac1{p^{z_1}}\right) \left(1-\frac1{p^{z_2}}\right)+ \left(1-\frac1{p^{z_3}}\right) - \left(1-\frac1{p^{z_1}}\right)\left(1-\frac1{p^{z_2}}\right)\left(1-\frac1{p^{z_3}}\right)\right)
\end{equation*}
\begin{equation*}
=  \prod_p \left(1- \frac1{p^{z_1+z_3}} -\frac1{p^{z_2+z_3}} + \frac1{p^{z_1+z_2+z_3}}\right),
\end{equation*}
absolutely convergent if $\Re z_j \ge 0$ \textup{($1\le j\le 3$)} and $\Re (z_1+z_3) >1$, $\Re (z_2+z_3) >1$.

If $z_1=\cdots =z_k=1$, then
\begin{equation*}
H_{k,s}(1,\ldots,1)= \prod_p  \left( \left(1-\frac1{p}\right)^s + \left(1-\frac1{p} \right)^{k-s} - \left(1-\frac1{p}\right)^k \right)
\end{equation*}
\begin{equation*}
= \prod_p  \left(1- \frac{p F_{k,s}(p)+(-1)^k}{p^k}\right),
\end{equation*}
by a direct computation.
\end{proof}

\begin{lemma} \label{Lemma_convo} Let $1\le s \le k-1$. For every $n_1,\ldots,n_k\in \N$,
\begin{equation} \label{lalambda}
\vartheta_{k,s}(n_1,\ldots,n_k) = \sum_{d_1\divides n_1, \ldots, d_k\divides
n_k} \lambda_{k,s}(d_1,\ldots,d_k),
\end{equation}
where the function $\lambda_{k,s}$ is multiplicative, and for any prime powers $p^{\nu_1},\ldots,p^{\nu_k}$
\textup{($\nu_1,\ldots,\nu_k\in \N_0$)}
\begin{equation} \label{lambda_val}
\lambda_{k,s}(p^{\nu_1},\ldots,p^{\nu_k})=
\begin{cases} 1, & \nu_1=\cdots=\nu_k=0,\\
(-1)^{\nu_1+\cdots+\nu_k-1}, & \nu_1,\ldots,\nu_k\in \{0,1\}, \\\
\ & \nu_1+\cdots+\nu_s\ge 1, \nu_{s+1}+\cdots +\nu_k\ge 1, \\
0, & \text{otherwise}.
\end{cases}
\end{equation}
\end{lemma}

\begin{proof} To deduce \eqref{lalambda}, observe that according to \eqref{product_repr}, the multiplicative function $\vartheta_{k,s}$ is the Dirichlet convolution of the 
constant $1$ function and the function $\lambda_{k,s}$, where 
\begin{equation*}
H_{k,s}(z_1,\ldots,z_k)= \sum_{n_1,\ldots,n_k=1}^{\infty} \frac{\lambda_{k,s}(n_1,\ldots,n_k)}{n_1^{z_1}\cdots n_k^{z_k}}.
\end{equation*}

The function $\vartheta_{k,s}$ is multiplicative, hence $\lambda_{k,s}$ is also multiplicative, 
and 
\begin{equation} \label{H_product}
H_{k,s}(z_1,\ldots,z_k) =\prod_p \sum_{\nu_1,\ldots,\nu_k=0}^{\infty}
\frac{\lambda_{k,s}(p^{\nu_1},\ldots,p^{\nu_k})}{p^{\nu_1 z_1+\cdots +\nu_k z_k}}.
\end{equation}

Identifying the terms of the sums under the products over the primes $p$ in \eqref{H_product} and \eqref{H_k_s}, respectively, 
we conclude \eqref{lambda_val}. Also see the example given in the proof of 
Lemma \ref{Lemma_Dir_series}.
\end{proof}

We will need the following properties of the function $\lambda_{k,s}$.

\begin{lemma} \label{Lemma_prop_lambda}
a) Let $\nu_1,\ldots,\nu_k\ge 0$. Then $\lambda_{k,s}(p^{\nu_1},\ldots,p^{\nu_k})=0$ provided that (i) $\nu_i\ge 2$ for at least one $i$ such that $1\le i\le k$, or (ii) $\nu_i=1$ for
at least one $i$ such that $1\le i\le s$ and $\nu_{s+1}=\cdots=\nu_k=0$, or (iii) $\nu_i=1$ for at least one $i$ with $s+1\le i\le k$ and $\nu_1=\cdots =\nu_s=0$. 

b) $\lambda_{k,s}(p,\ldots,p)=(-1)^{k-1}$ for every prime $p$.

c) Let $j_1,\ldots,j_k,\delta\in \N$. If $\gcd(j_1\cdots j_k,\delta)\ne 1$, then  $\lambda_{k,s}(j_1\delta,\ldots,j_k\delta)=0$. If $\gcd(j_1\cdots j_k,\delta) = 1$, then $\lambda_{k,s}(j_1\delta,\ldots,j_k\delta) = \lambda_{k,s}(j_1,\ldots,j_k) \lambda_{k,s}(\delta,\ldots,\delta)$. 

d) Let $\delta\in \N$. Then $\lambda_{k,s}(\delta,\ldots,\delta) =(-1)^{(k-1)\omega(\delta)}\mu^2(\delta)$.
\end{lemma}

\begin{proof} a), b) Follow from \eqref{lambda_val}.

c) If $\gcd(j_1\cdots j_k,\delta)\ne 1$, then at least one $j_i\delta$ is not squarefree and $\lambda_{k,s}(j_1\delta,\ldots,j_k\delta)=0$ by property a). If $\gcd(j_1\cdots j_k,\delta) = 1$, then $\lambda_{k,s}(j_1\delta,\ldots,j_k\delta) = \lambda_{k,s}(j_1,\ldots,j_k) \lambda_{k,s}(\delta,\ldots,\delta)$ by the multiplicativity of the function $\lambda_{k,s}$. 

d) If $\delta$ is not squarefree, then $\lambda_{k,s}(\delta,\ldots,\delta)=0$ by a). If $\delta=p_1\cdots p_r$ is squarefree, then using multiplicativity and b) we have $\lambda_{k,s}(\delta,\ldots,\delta) = \lambda_{k,s}(p_1,\ldots,p_1) \cdots \lambda_{k,s}(p_r,\ldots,p_r)= (-1)^{(k-1)\omega(\delta)}$.
\end{proof}

\begin{lemma} \label{Lemma_Repr_A_N} Let $1 \le s \le k-1$. For every $n\ge k$ we have
\begin{equation} \label{form_A_k_t}
A_{k,s}(n) = \sum_{\delta \divides n} \mu^2(\delta) (-1)^{(k-1)\omega(\delta)} \sum_{\substack{1\le j_1,\ldots,j_k\le n/\delta \\ \gcd(j_1,\ldots,j_k)=1\\
\gcd(j_1\cdots j_k, \delta)=1 }}  \lambda_{k,s}(j_1,\ldots,j_k) N(n/\delta;j_1,\ldots,j_k).
\end{equation}
\end{lemma}

\begin{proof} By the definitions of $A_{k,s}(n)$ and $\vartheta_{k,s}(n_1,\ldots,n_k)$, and by Lemma \ref{Lemma_convo} we have
\begin{equation*}
A_{k,s}(n)= \sum_{\substack{x_1,\ldots,x_k\ge 1\\x_1+\cdots+ x_k=n}} \vartheta_{k,s}(x_1,\ldots,x_k)
= \sum_{\substack{x_1,\ldots,x_k\ge 1\\ x_1+\cdots+ x_k=n}} \sum_{d_1\divides x_1, \ldots, d_k\divides
x_k} \lambda_{k,s}(d_1,\ldots,d_k)
\end{equation*}
\begin{equation*}
= \sum_{1\le d_1,\ldots,d_k\le n} \lambda_{k,s}(d_1,\ldots,d_k)
\sum_{\substack{a_1,\ldots,a_k\ge 1\\ d_1a_1+\cdots+ d_ka_k=n}} 1
\end{equation*}
\begin{equation*}
= \sum_{1\le d_1,\ldots,d_k\le n}  \lambda_{k,s}(d_1,\ldots,d_k) N(n;d_1,\ldots,d_k).
\end{equation*}

Let $\delta=\gcd(d_1,\ldots,d_k)$. If $\delta \notdivides n$, then the equation $d_1a_1+\cdots+ d_ka_k=n$ has no solutions in $a_1,\ldots,a_k$. If
$\delta \divides n$, then $N(n;d_1,\ldots,d_k)= N(n/\delta;d_1/\delta,\ldots,d_k/\delta)$. We deduce, by grouping the terms according to the values
of $\delta$,
\begin{equation*}
A_{k,s}(n) = \sum_{\delta \divides n} \sum_{\substack{1\le d_1,\ldots,d_k\le n\\ \gcd(d_1,\ldots,d_k)=\delta}}  \lambda_{k,s}(d_1,\ldots,d_k)
N(n/\delta;d_1/\delta,\ldots,d_k/\delta)
\end{equation*}
\begin{equation} \label{sum_for_A}
= \sum_{\delta \divides n} \sum_{\substack{1\le j_1,\ldots,j_k\le n/\delta \\ \gcd(j_1,\ldots,j_k)=1}}  \lambda_{k,s}(j_1\delta,\ldots,j_k\delta)
N(n/\delta;j_1,\ldots,j_k).
\end{equation}

Now we use properties c) and d) of Lemma \ref{Lemma_prop_lambda}. If $\gcd(j_1\cdots j_k,\delta)\ne 1$, then $\lambda_{k,s}(j_1\delta,\ldots,j_k\delta)=0$.
If $\gcd(j_1\cdots j_k,\delta) = 1$, then $\lambda_{k,s}(j_1\delta,\ldots,j_k\delta) = \lambda_{k,s}(j_1,\ldots,j_k) \lambda_{k,s}(\delta,\ldots,\delta)$, where $\lambda_{k,s}(\delta,\ldots,\delta)=(-1)^{(k-1)\omega(\delta)}\mu^2(\delta)$.

Inserting these into \eqref{sum_for_A} gives formula \eqref{form_A_k_t}.
\end{proof}

\begin{lemma} \label{Lemma_T_delta} If $1\le s \le k-1$, then for every $\delta \in \N$,
\begin{equation*}
T_{\delta}:= \sum_{\substack{j_1,\ldots,j_k=1 \\ \gcd(j_1,\ldots,j_k)=1\\ \gcd(j_1\cdots j_k,\delta)=1 }}^{\infty}
\frac{\lambda_{k,s}(j_1,\ldots,j_k)}{j_1\cdots j_k}
=  \prod_{p\,\centernot\divides \, \delta} \left(1 - \frac{F_{k,s}(p)}{p^{k-1}} \right).
\end{equation*}
\end{lemma}

\begin{proof} We have
\begin{equation*}
T_{\delta}= \sum_{\substack{j_1,\ldots,j_k=1 \\ \gcd(j_1\cdots j_k,\delta)=1 }}^{\infty}
\frac{\lambda_{k,s}(j_1,\ldots,j_k)}{j_1\cdots j_k} \sum_{c \divides \gcd(j_1,\ldots,j_k)} \mu(c)
\end{equation*}
\begin{equation*}
 = \sum_{\substack{c,v_1,\ldots,v_k=1 \\ \gcd(v_1\cdots v_k,\delta)=1\\ \gcd(c,\delta)=1}}^{\infty}
\frac{\mu(c) \lambda_{k,s}(cv_1,\ldots,cv_k)}{cv_1\cdots cv_k}.
\end{equation*}

Here we can assume that $c$ is squarefree and we use property c) of Lemma \ref{Lemma_prop_lambda} (with $c$ instead of $\delta$ and with $v_1,\ldots,v_k$ instead of $j_1,\ldots,j_k$). We obtain that

\begin{equation*}
T_{\delta}= \sum_{\substack{c=1\\ (c,\delta)=1}}^{\infty} \frac{\mu(c)\lambda_{k,s}(c,\ldots,c)}{c^k}
\sum_{\substack{v_1,\ldots,v_k=1 \\ \gcd(v_1\cdots v_k,c\delta)=1}}^{\infty}
\frac{\lambda_{k,s}(v_1,\ldots,v_k)}{v_1\cdots v_k}.
\end{equation*}

Now from Lemma \ref{Lemma_Dir_series} we have
\begin{equation*}
T_{\delta} = \sum_{\substack{c=1\\ (c,\delta)=1}}^{\infty} \frac{\mu(c)\lambda_{k,s}(c,\ldots,c)}{c^k} \prod_{p\, \centernot\divides \, c\delta}
\left(1- \frac{p F_{k,s}(p)+(-1)^k}{p^k} \right)
\end{equation*}
\begin{equation*}
= \prod_{p\, \centernot\divides \, \delta} \left(1- \frac{p F_{k,s}(p)+(-1)^k}{p^k} \right)  \sum_{\substack{c=1\\ (c,\delta)=1}}^{\infty} \frac{\mu(c)\lambda_{k,s}(c,\ldots,c)}{c^k} \prod_{p \divides c} \left(1- \frac{p F_{k,s}(p)+(-1)^k}{p^k} \right)^{-1}
\end{equation*}
\begin{equation*}
= \prod_{p\, \centernot\divides \, \delta} \left(1- \frac{p F_{k,s}(p)+(-1)^k}{p^k} \right) \prod_{p\, \centernot\divides \, \delta}
\left(1- \frac{\lambda_{k,s}(p,\ldots,p)}{p^k \left(1- \frac{p F_{k,s}(p)+(-1)^k}{p^k}  \right) }\right)
= \prod_{p\, \centernot\divides \, \delta}  \left(1 - \frac{F_{k,s}(p)}{p^{k-1}} \right),
\end{equation*}
using that $\lambda_{k,s}(p,\ldots,p)=(-1)^{k-1}$ (see property b) of Lemma \ref{Lemma_prop_lambda}).
\end{proof}

\begin{proof}[Proof of Theorem {\rm \ref{Theorem_A}}]

By applying Lemma \ref{Lemma_Partition} we have
\begin{equation*}
N(n/\delta;j_1,\ldots,j_k)= \frac{(n/\delta)^{k-1}}{(k-1)!j_1\cdots j_k}+  O\left(\frac{(n/\delta)^{k-2}(j_1+\cdots +j_k)}{j_1\cdots j_k} \right),
\end{equation*}
and inserting it into \eqref{form_A_k_t} gives
\begin{equation} \label{A_to_go}
A_{k,s}(n) = \frac{n^{k-1}}{(k-1)!}  \sum_{\delta \divides n} \frac{\mu^2(\delta)}{\delta^{k-1}} (-1)^{(k-1)\omega(\delta)} \sum_{\substack{j_1,\ldots,j_k\le n/\delta \\ \gcd(j_1,\ldots,j_k)=1\\ \gcd(j_1\cdots j_k,\delta)=1 }} 
\frac{\lambda_{k,s}(j_1,\ldots,j_k)}{j_1\cdots j_k} 
\end{equation}
\begin{equation} \label{A_error}
+ O\left(n^{k-2} \sum_{\delta \divides n} \frac{\mu^2(\delta)}{\delta^{k-2}} \sum_{j_1,\ldots,j_k\le n/\delta}
\frac{(j_1+\cdots +j_k)|\lambda_{k,s}(j_1,\ldots,j_k)|}{j_1\cdots j_k} \right).
\end{equation} 

The inner sum $I(n,\delta)$ in \eqref{A_to_go} can be written as
\begin{equation}  \label{def_I_R}
I(n,\delta): = \sum_{\substack{j_1,\ldots,j_k\le n/\delta \\ \gcd(j_1,\ldots,j_k)=1\\ \gcd(j_1\cdots j_k,\delta)=1 }} 
\frac{\lambda_{k,s}(j_1,\ldots,j_k)}{j_1\cdots j_k}  = \sum_{\substack{j_1,\ldots,j_k =1 \\ \gcd(j_1,\ldots,j_k)=1\\ 
\gcd(j_1\cdots j_k,\delta)=1 }}^{\infty} \frac{\lambda_{k,s}(j_1,\ldots,j_k)}{j_1\cdots j_k} + R(n,\delta), 
\end{equation}
where for the remainder term $R$ one has
\begin{equation} \label{first_est_R}
R(n,\delta) \ll  \sideset{}{'}\sum_{j_1,\ldots,j_k} \frac{|\lambda_{k,s}(j_1,\ldots,j_k)|}{j_1\cdots j_k}
\end{equation}
with $\sum^{'}$ meaning that $j_1,\ldots, j_k\in \N$ and $j_1,\ldots,j_k\le n/\delta$ does not hold, that is, there exists at least one $i$ 
($1\le i\le k$) such that $j_i>n/\delta$ (the gcd conditions can be omitted).

Hence, according to \eqref{A_to_go}, \eqref{def_I_R} and using Lemma \ref{Lemma_T_delta} the main term for $A_{k,s}(n)$ is 
\begin{equation*}
\frac{n^{k-1}}{(k-1)!} \prod_p  \left(1 - \frac{F_{k,s}(p)}{p^{k-1}} \right) \sum_{\delta \divides n} \frac{\mu^2(\delta)}{\delta^{k-1}}
(-1)^{(k-1)\omega(\delta)} \prod_{p\divides \delta}  \left(1 - \frac{F_{k,s}(p)}{p^{k-1}} \right)^{-1}
\end{equation*}
\begin{equation*}
= \frac{n^{k-1}}{(k-1)!}  \prod_p  \left(1 - \frac{F_{k,s}(p)}{p^{k-1}} \right)
\prod_{p\divides n} \left(1 + \frac{(-1)^{k-1}}{p^{k-1}}  \left(1 - \frac{F_{k,s}(p)}{p^{k-1}} \right)^{-1} \right)
\end{equation*}
\begin{equation*}
= \frac{n^{k-1}}{(k-1)!}  \prod_p  \left(1 - \frac{F_{k,s}(p)}{p^{k-1}} \right)
\prod_{p\divides n} \left(1 + \frac{(-1)^{k-1}}{p^{k-1}- F_{k,s}(p)} \right)
\end{equation*}
\begin{equation*}
=\frac{n^{k-1}}{(k-1)!} C_{k,s} f_{k,s}(n).
\end{equation*}

To estimate the error under \eqref{A_error} note that for any $i$ ($1\le i\le k$),
\begin{equation*}
S:= \sum_{j_1,\ldots,j_k\le x} \frac{j_i|\lambda_{k,s}(j_1,\ldots,j_k)|}{j_1\cdots j_k} =
\sum_{j_1,\ldots,j_k\le x} \frac{|\lambda_{k,s}(j_1,\ldots,j_k)|}{j_1\cdots j_{i-1} j_{i+1}\cdots j_k}
\end{equation*}
\begin{equation*}
\le \prod_{p\le x} \sum_{\nu_1,\ldots,\nu_k=0}^{\infty} \frac{|\lambda_{k,s}(p^{\nu_1},\ldots,p^{\nu_k})|}{p^{\nu_1+\cdots
+\nu_{i-1}+\nu_{i+1}+\cdots + \nu_k}}
\end{equation*}
\begin{equation*}
= \prod_{p\le x} \left(1+\frac{c_1}{p}+ \frac{c_2}{p^2}+\cdots +\frac{c_{k-1}}{p^{k-1}}\right),
\end{equation*}
by \eqref{lambda_val}, where $c_1,\ldots,c_{k-1}$ are certain positive
integers. We need the precise value of $c_1$. If $1\le i\le s$, then we have $p$ in the denominator if and only
if $\nu_i=1$ and exactly one of $\nu_{s+1},\ldots,\nu_k$ is $1$, all other values $\nu_j$ being $0$, which occurs $k-s$ times.
That is, $c_1=k-s$. In the case $s+1\le i\le k$, the denominator is $p$ if and only
if $\nu_i=1$ and exactly one of $\nu_1,\ldots,\nu_s$ is $1$, all other
values $\nu_j$ being $0$, which occurs $s$ times. So, in this case $c_1=s$.

We deduce that
\begin{equation} \label{estimate_S_i}
S\ll  \prod_{p\le x} \left(1+\frac1{p}\right)^{\max(s,k-s)} \ll (\log x)^{\max(s,k-s)},
\end{equation}
by Mertens' theorem. For $x=n/\delta$ we obtain that the error in \eqref{A_error} is
\begin{equation} \label{result_R_2}
 \ll n^{k-2} \sum_{\delta \divides n} \frac{\mu^2(\delta)}{\delta^{k-2}} (\log n/\delta)^{\max(s,k-s)} \le n^{k-2} (\log n)^{\max(s,k-s)}
\prod_{p\divides n} \left(1+\frac1{p^{k-2}} \right),
\end{equation}
where $P_k:= \prod_{p\divides n} \left(1+\frac1{p^{k-2}}\right) \ll 1$ if $k\ge 4$ and $P_3=\prod_{p\divides n} \left(1+\frac1{p} \right)=
\frac{n}{\varphi(n)} \prod_{p\divides n} \left(1-\frac1{p^2} \right) \ll \log\log n$ if $k=3$.

Now we estimate the error $R(n,\delta)$ from \eqref{first_est_R}. Consider the sums
\begin{equation*}
\sideset{}{'}\sum_{j_1,\ldots,j_k} \frac{|\lambda_{k,s}(j_1,\ldots,j_k)|}{j_1\cdots j_k}
\end{equation*}
with $n/\delta=x$, and assume that $j_i>x$  for $i\in I\subseteq \{1,\ldots,k\}$ and $j_{\ell}\le x$ for $\ell \notin I$. That is, we need to
estimate the sum
\begin{equation*}
U:= \sum_{\substack{j_i>x, \, i\in I\\ j_{\ell}\le x, \, \ell \notin I}} \frac{|\lambda_{k,s}(j_1,\ldots,j_k)|}{j_1\cdots j_k}.
\end{equation*}

Let $m=\# I$. We distinguish the following cases:

Case i) $m\ge 3$: if $0<\varepsilon <1/2$, then
\begin{equation*}
U = \sum_{\substack{j_i>x, \, i\in I\\ j_{\ell}\le x, \, \ell \notin I}} \frac{|\lambda_{k,s}(j_1,\ldots,j_k)|
\prod_{i\in I} j_i^{\varepsilon-1/2}}{\prod_{i\in I} j_i^{\varepsilon+1/2} \prod_{\ell \notin I} j_{\ell}}
\end{equation*}
\begin{equation*}
\le x^{m(\varepsilon-1/2)} \sum_{j_1,\ldots,j_k=1}^{\infty} \frac{|\lambda_{k,s}(j_1,\ldots,j_k)|}
{\prod_{i\in I} j_i^{\varepsilon+1/2} \prod_{\ell \notin I} j_{\ell}}\ll x^{m(\varepsilon-1/2)},
\end{equation*}
since the latter series is convergent by Lemma \ref{Lemma_Dir_series}. Using that
$m(\varepsilon -1/2)< -1$ for $0<\varepsilon< (m-2)/(2m)$, here we
need $m\ge 3$, we obtain that $U\ll \frac1{x}$.

Case ii) $m=1$: Assume that $I=\{a\}$, $j_a>x$ and $j_{\ell} \le x$ for all $\ell \ne a$. Consider a prime $p$.
If $p\divides j_{\ell}$ for some $\ell \ne a$, then $p\le x$. If $p\divides j_a$ and $p>x$, then $p\notdivides j_{\ell}$ for every
$\ell \ne a$ and $\lambda_{k,s}(j_1,\ldots,j_k)=0$ by its definition \eqref{lambda_val}. Hence it is enough to select the primes
$p\le x$. We deduce
\begin{equation*}
U < \frac1{x} \sum_{\substack{j_a > x\\ j_{\ell} \le x, \, \ell \ne a}}
\frac{|\lambda_{k,s}(j_1,\ldots,j_k)|}{j_1\cdots j_{a-1}j_{a+1}\cdots j_k}
\end{equation*}
\begin{equation*}
\le \frac1{x} \prod_{p\le x} \sum_{\nu_1,\ldots, \nu_k=0}^{\infty}
\frac{|\lambda_{k,s}(p^{\nu_1},\ldots,p^{\nu_k})|}{p^{\nu_1+\cdots +\nu_{a-1}+\nu_{a+1}+ \cdots + \nu_k}} \ll \frac1{x}(\log x)^{\max(s,k-s)},
\end{equation*}
similar to the estimate \eqref{estimate_S_i}.

Case iii) $m=2$: Assume that $I=\{a,b\}$, that is, $j_a,j_b>x$ and $j_{\ell} \le x$ for all $\ell \ne a,b$. Furthermore, suppose that
$1\le a,b\le s$ or $s+1\le a,b\le k$. Let $p$ be a prime. If $p\divides j_{\ell}$ for some $\ell \ne a,b$, then $p\le x$. If $p\divides j_a$
or $p\divides j_b$, and $p>x$, then $p\notdivides j_{\ell}$ for every $\ell \ne a,b$ and $\lambda_{k,s}(j_1,\ldots,j_k)=0$ by \eqref{lambda_val}.
So, it is enough to consider the primes $p\le x$. Similar to the case ii) we deduce that
\begin{equation*}
U < \frac1{x} \sum_{\substack{j_a, j_b > x\\ j_{\ell} \le x, \, \ell \ne a,b}}
\frac{|\lambda_{k,s}(j_1,\ldots,j_k)|}{j_1\cdots j_{a-1}j_{a+1}\cdots j_k}
\end{equation*}
\begin{equation*}
\le \frac1{x} \prod_{p\le x} \sum_{\nu_1,\ldots, \nu_k=0}^{\infty}
\frac{|\lambda_{k,s}(p^{\nu_1},\ldots,p^{\nu_k})|}{p^{\nu_1+\cdots +\nu_{a-1}+\nu_{a+1}+ \cdots + \nu_k}} \ll \frac1{x}(\log x)^{\max(s,k-s)}.
\end{equation*}

Hence, we need to handle one more case, namely when $j_a,j_b>x$, where $1\le a\le s$, $s+1\le b\le k$ and $j_{\ell} \le x$ for all $\ell \ne a,b$.
We split the sum $U$ into two sums, namely
\begin{equation*}
U = \sum_{\substack{j_a, j_b>x \\ j_{\ell} \le x, \, \ell \ne a,b}}
\frac{|\lambda_{k,s}(j_1,\ldots,j_k)|}{j_1\cdots j_k}
\end{equation*}
\begin{equation*}
= \sum_{\substack{j_a> x^{3/2}, \, j_b>x\\ j_{\ell} \le x, \, \ell \ne a,b}}
\frac{|\lambda_{k,s}(j_1,\ldots,j_k)|}{j_1\cdots j_k} +
\sum_{\substack{x^{3/2}\ge j_a> x, \, j_b>x\\ j_{\ell} \le x, \, \ell \ne a,b }}
\frac{|\lambda_{k,s}(j_1,\ldots,j_k)|}{j_1\cdots j_k} =: U_1+U_2,
\end{equation*}
say, where
\begin{equation*}
U_1= \sum_{\substack{j_a> x^{3/2}, \, j_b >x\\ j_{\ell} \le x, \, \ell \ne a,b }}
\frac{|\lambda_{k,s}(j_1,\ldots,j_k)|}{j_a^{1/3} \prod_{\ell \ne a} j_{\ell}}
\frac1{j_a^{2/3}}
\end{equation*}
\begin{equation*}
\le \frac1{x} \sum_{\substack{j_1,\ldots,j_k=1}}^{\infty}
\frac{|\lambda_{k,s}(j_1,\ldots,j_k)|}{j_a^{1/3} \prod_{\ell \ne a} j_{\ell}}
\ll \frac1{x},
\end{equation*}
since the series is convergent.

At the same time,
\begin{equation*}
U_2 \le \frac1{x} \sum_{\substack{x^{3/2}\ge j_a, \, j_b>x \\ j_{\ell} \le x, \, \ell \ne a,b }}
\frac{|\lambda_{k,s}(j_1,\ldots,j_k)|}{\prod_{1\le \ell \le k, \ell \ne b} j_{\ell}}
\end{equation*}
where $j_a\le x^{3/2}$, $j_b>x$, $j_{\ell} \le x$ ($\ell \ne a, b$).
Let $p$ be a  prime. If $p\divides j_{\ell}$ for an $\ell \in \{1,\ldots,k\}\setminus \{b\} $,
then $p\le x^{3/2}$. If $p\divides j_b$ and $p>x^{3/2}$, then $p\notdivides
j_{\ell}$ for every $\ell \ne b$ and $\lambda_{k,s}(j_1,\ldots,j_k)=0$ by its definition. Hence it is enough
to take the primes $p\le x^{3/2}$. We deduce, as above,
\begin{equation*}
U_2 \le \frac1{x} \prod_{p\le x^{3/2}} \sum_{\nu_1,\ldots, \nu_k=0}^{\infty}
\frac{|\lambda_{k,s}(p^{\nu_1},\ldots,p^{\nu_k})|}{p^{\nu_1+\cdots +\nu_{b-1}+\nu_{b+1} +\cdots
+\nu_k}}
\end{equation*}
\begin{equation*}
\ll \frac1{x}(\log x^{3/2})^{\max(s,k-s)} \ll \frac1{x}(\log
x)^{\max(s,k-s)}.
\end{equation*}

Hence for every $m\ge 1$,
\begin{equation*}
U \ll \frac1{x}(\log x)^{\max(s,k-s)},
\end{equation*}
which applied for $x=n/\delta$ gives by \eqref{A_to_go}, \eqref{def_I_R} and \eqref{first_est_R} that the error coming from $R(n,\delta)$ is
\begin{equation*}
\ll n^{k-1} \sum_{\delta \mid n} \frac{\mu^2(\delta)}{\delta^{k-1}} R(n,\delta) \ll n^{k-2} \sum_{\delta \divides n} \frac{\mu^2(\delta)}{\delta^{k-2}} (\log n/\delta)^{\max(s,k-s)},
\end{equation*}
leading to the same estimate as that for the error \eqref{A_error} obtained in \eqref{result_R_2}.

This finishes the proof of Theorem \ref{Theorem_A}.
\end{proof}

\subsection{The sum $B_{k,t}(n)$}

Similar to the proof of Theorem \ref{Theorem_A}, consider the characteristic function of the $k$-tuples of positive integers with
$t$-wise relatively prime components, i.e., let
\begin{equation*}
\varrho_{k,t}(n_1,\ldots,n_k)= \begin{cases} 1, & \text{ if $\gcd(n_{i_1},\ldots n_{i_t})=1$ for $1\le i_1<\ldots <i_t\le k$},
\\ 0, &
\text{ otherwise.}
\end{cases}
\end{equation*}

The function $\varrho_{k,t}(n_1,\ldots,n_k)$ is symmetric in all variables and it is a multiplicative function of $k$ variables.
For every prime $p$ and every $\nu_1,\ldots,\nu_r\in \N_0$ one has
\begin{equation*}
\varrho_{k,t}(p^{\nu_1}, \ldots,p^{\nu_k})= \begin{cases} 1, &
\text{if there are at most $t-1$ values $\nu_i\ge 1$}, \\ 0, &
\text{otherwise.}
\end{cases}
\end{equation*}

The following result on the multiple Dirichlet series of the function $\varrho_{k,t}$ was
proved by the author \cite{Tot2016}. Let $e_j(x_1,\dots,x_k)=\sum_{1\le i_1<\ldots<i_j\le k}
x_{i_1}\cdots x_{i_j}$ denote the elementary symmetric polynomials
in $x_1,\ldots,x_k$ of degree $j$.

\begin{lemma}[{\cite[Th.\ 2.1]{Tot2016}}] \label{Lemma_Dir_series_rho} Let $2\le t\le k$. Then
\begin{equation*}
\sum_{n_1,\ldots,n_k=1}^{\infty} \frac{\varrho_{k,t}(n_1,\ldots,n_k)}{n_1^{z_1}\cdots n_r^{z_k}}=
\zeta(z_1)\cdots \zeta(z_k) L_{k,t}(z_1,\ldots,z_k),
\end{equation*}
where
\begin{equation*}
L_{k,t}(z_1,\ldots,z_k)= \prod_p \left( 1- \sum_{j=t}^k (-1)^{j-t} \binom{j-1}{t-1} e_j(p^{-z_1},\ldots,p^{-z_k})\right)
\end{equation*}
is absolutely convergent if $\Re (z_{i_1}+\cdots+ z_{i_j})>1$ for every $1\le i_1< \ldots < i_j\le k$ with $t\le j\le
k$.

In particular, for $z_1=\cdots =z_k=1$,
\begin{equation*}
L_{k,t}(1,\ldots,1)= \prod_p  \frac{pG_{k,t}(p)+(-1)^{k-t+1}\binom{k-1}{t-1}}{p^k}
\end{equation*}
is absolutely convergent, where $G_{k,t}(p)$ is defined by \eqref{notation_G}.
\end{lemma}

\begin{lemma} \label{Lemma_convo_rho} Let $2\le t\le k$. Then for every $n_1,\ldots,n_k\in \N$,
\begin{equation*}
\varrho_{k,t}(n_1,\ldots,n_k) = \sum_{d_1\divides n_1, \ldots, d_k\divides
n_k} \psi_{k,t}(d_1,\ldots,d_k),
\end{equation*}
where the function $\psi_{k,t}$ is multiplicative, and for any prime powers $p^{\nu_1},\ldots,p^{\nu_k}$ \textup{($\nu_1,\ldots,\nu_k\in \N_0$)},
\begin{equation} \label{psi}
\psi_{k,t}(p^{\nu_1},\ldots,p^{\nu_k})=
\begin{cases} 1, & \nu_1=\ldots=\nu_k=0,\\
(-1)^{j-t+1}\binom{j-1}{t-1}, & \nu_1,\ldots,\nu_k\in \{0,1\}, \ j:=\nu_1+\ldots+\nu_k \ge t,\\
0, & \text{otherwise}.
\end{cases}
\end{equation}
\end{lemma}

\begin{proof} This follows by Lemma \ref{Lemma_Dir_series_rho}, where
\begin{equation*}
L_{k,t}(z_1,\ldots,z_k)= \sum_{n_1,\ldots,n_r=1}^{\infty} \frac{\psi_{k,t}(n_1,\ldots,n_k)}{n_1^{z_1}\cdots n_k^{z_k}}.
\end{equation*}

See the proof of the similar results included in Lemma \ref{Lemma_convo}.
\end{proof}

We will need the following properties the function $\psi_{k,t}$, given by \eqref{psi}, which are similar to the function $\lambda_{k,s}$. Cf. Lemma \ref{Lemma_prop_lambda}.

\begin{lemma} \label{Lemma_prop_psi} Let $j_1,\ldots,j_k,\delta\in \N$.

a) If $\gcd(j_1\cdots j_k,\delta)\ne 1$, then  $\psi_{k,t}(j_1\delta,\ldots,j_k\delta)=0$. If $\gcd(j_1\cdots j_k,\delta) = 1$, then 
\begin{equation*}
\psi_{k,t}(j_1\delta,\ldots,j_k\delta) = \psi_{k,t}(j_1,\ldots,j_k) \psi_{k,t}(\delta,\ldots,\delta).
\end{equation*}

b) $\psi_{k,t}(\delta,\ldots,\delta) =\left((-1)^{k-t+1}\binom{k-1}{t-1}\right)^{\omega(\delta)}\mu^2(\delta)$.
\end{lemma}

\begin{lemma} \label{Lemma_Repr_B_N} Let $2 \le t \le k$. For every $n\ge k$ we have
\begin{equation} \label{form_B_k_t}
B_{k,t}(n) = \sum_{\delta \divides n} \mu^2(\delta) \left((-1)^{k-t+1} \binom{k-1}{t-1}\right)^{\omega(\delta)}
\sum_{\substack{1\le j_1,\ldots,j_k\le n/\delta \\ \gcd(j_1,\ldots,j_k)=1\\ \gcd(j_1\cdots j_k, \delta)=1 }}  \psi_{k,t}(j_1,\ldots,j_k)
N(n/\delta;j_1,\ldots,j_k).
\end{equation}
\end{lemma}

\begin{proof} We follow the same steps as in the proof of Lemma \ref{Lemma_Repr_A_N}. First, by the definitions of $B_{k,t}(n)$ and
$\varrho_{k,t}(n_1,\ldots,n_k)$,
and by Lemma \ref{Lemma_convo_rho} we have
\begin{equation*}
B_{k,t}(n)= \sum_{\substack{x_1,\ldots,x_k\ge 1\\x_1+\cdots+ x_k=n}} \varrho_{k,t}(x_1,\ldots,x_k)
= \sum_{\substack{x_1,\ldots,x_k\ge 1\\ x_1+\cdots+ x_k=n}} \sum_{d_1\divides x_1, \ldots, d_k\divides x_k}
\psi_{k,t}(d_1,\ldots,d_k)
\end{equation*}
\begin{equation*}
= \sum_{1\le d_1,\ldots,d_k\le n} \psi_{k,t}(d_1,\ldots,d_k)
\sum_{\substack{a_1,\ldots,a_k\ge 1\\ d_1a_1+\cdots+ d_ka_k=n}} 1
\end{equation*}
\begin{equation*}
= \sum_{1\le d_1,\ldots,d_k\le n}  \psi_{k,t}(d_1,\ldots,d_k) N(n;d_1,\ldots,d_k).
\end{equation*}

By grouping the terms according to the values of $\delta=\gcd(a_1,\ldots,a_k)$,
\begin{equation} \label{sum_for_B}
B_{k,t}(n) = \sum_{\delta \divides n} \sum_{\substack{1\le j_1,\ldots,j_k\le n/\delta \\ \gcd(j_1,\ldots,j_k)=1}}
\psi_{k,t}(j_1\delta,\ldots,j_k\delta)
N(n/\delta;j_1,\ldots,j_k).
\end{equation}

By using Lemma \ref{Lemma_prop_psi}  we obtain identity \eqref{form_B_k_t}.
\end{proof}

\begin{lemma} \label{Lemma_V_delta} If $2 \le t \le k$, then for every $\delta \in \N$,
\begin{equation*}
V_{\delta}: =\sum_{\substack{j_1,\ldots,j_k=1 \\ \gcd(j_1,\ldots,j_k)=1\\ \gcd(j_1\cdots j_k,\delta)=1 }}^{\infty}
\frac{\psi_{k,t}(j_1,\ldots,j_k)}{j_1\cdots j_k}
=  \prod_{p\, \centernot\divides \, \delta} \frac{G_{k,t}(p)}{p^{k-1}},
\end{equation*}
\end{lemma}

\begin{proof} We have by using the property of the M\"{o}bius function and the properties of the function $\psi_{k,t}$ in  Lemma \ref{Lemma_prop_psi},
\begin{equation*}
V_{\delta}= \sum_{\substack{j_1,\ldots,j_k=1 \\ \gcd(j_1\cdots j_k,\delta)=1 }}^{\infty}
\frac{\psi_{k,t}(j_1,\ldots,j_k)}{j_1\cdots j_k} \sum_{c\divides \gcd(j_1,\ldots,j_k)} \mu(c)
\end{equation*}
\begin{equation*}
 = \sum_{\substack{c,v_1,\ldots,v_k=1 \\ \gcd(cv_1\cdots v_k,\delta)=1 }}^{\infty}
\frac{\mu(c) \psi_{k,t}(cv_1,\ldots,cv_k)}{cv_1\cdots cv_k} =
\sum_{\substack{c=1\\ (c,\delta)=1}}^{\infty} \frac{\mu(c)\psi_{k,t}(c,\ldots,c)}{c^k}
\sum_{\substack{v_1,\ldots,v_k=1 \\ \gcd(v_1\cdots v_k,c\delta)=1}}^{\infty}
\frac{\psi_{k,t}(v_1,\ldots,v_k)}{v_1\cdots v_k}.
\end{equation*}

We deduce from Lemma
\ref{Lemma_Dir_series_rho} that
\begin{equation*}
V_{\delta} = \sum_{\substack{c=1\\ (c,\delta)=1}}^{\infty} \frac{\mu(c)\psi_{k,t}(c,\ldots,c)}{c^k} \prod_{p\, \centernot\divides \, c\delta}
\frac{pG_{k,t}(p)+(-1)^{k-t+1}\binom{k-1}{t-1}}{p^k}
\end{equation*}
\begin{equation*}
= \prod_{p\, \centernot\divides \, \delta} \frac{pG_{k,t}(p)+(-1)^{k-t+1}\binom{k-1}{t-1}}{p^k}  \sum_{\substack{c=1\\ (c,\delta)=1}}^{\infty}
\frac{\mu(c)\psi_{k,t}(c,\ldots,c)}{c^k} \prod_{p\divides c}  \left(\frac{pG_{k,t}(p)+(-1)^{k-t+1}\binom{k-1}{t-1}}{p^k}\right)^{-1}
\end{equation*}
\begin{equation*}
= \prod_{p\, \centernot\divides \, \delta} \frac{pG_{k,t}(p)+(-1)^{k-t+1}\binom{k-1}{t-1}}{p^k} \prod_{p\, \centernot\divides \, \delta}
\left(1- \frac{\psi_{k,t}(p,\ldots,p)}{p^k}\cdot  \frac{p^k}{pG_{k,t}(p)+(-1)^{k-t+1}\binom{k-1}{t-1}}\right)
\end{equation*}
\begin{equation*}
= \prod_{p\, \centernot\divides \, \delta} \frac{G_{k,t}(p)}{p^{k-1}},
\end{equation*}
using that $\psi_{k,t}(p,\ldots,p)=(-)^{k-t+1}\binom{k-1}{t-1}$. Cf. property b) of Lemma \ref{Lemma_prop_psi}.
\end{proof}

We also need the following estimates proved in \cite[Sect.\ 4]{Tot2016}, by similar arguments as given in the proof of Theorem \ref{Theorem_A}.

\begin{lemma} \label{Lemma_error_psi} Let $2\le t\le k$. Then for any $i$ \textup{($1\le i\le k$)},
\begin{equation} \label{estimate_1}
\sum_{j_1,\ldots,j_k\le x} \frac{|\psi_{k,t}(j_1,\ldots,j_k)|}{j_1\cdots j_{i-1}j_{i+1}\cdots j_k} \ll
\begin{cases} 1, & \text{ if $t\ge 3$}, \\ (\log x)^{k-1}, & \text{ if $t=2$}. \end{cases}
\end{equation}

Furthermore,
\begin{equation} \label{estimate_2}
\sideset{}{'}\sum_{j_1,\ldots,j_k} \frac{|\psi_{k,t}(j_1,\ldots,j_k)|}{j_1\cdots j_k} \ll
\begin{cases} x^{-1}, & \text{ if $t\ge 3$}, \\ x^{-1} (\log x)^{k-1}, & \text{ if $t=2$}, \end{cases}
\end{equation}
with $\sum^{'}$ meaning that there is at least one $i$ \textup{($1\le i\le k$)} such that $j_i>x$.
\end{lemma}

\begin{proof}[Proof of Theorem {\rm \ref{Theorem_B}}]
Identity \eqref{form_B_k_t} and Lemma \ref{Lemma_Partition} lead to
\begin{equation} \label{B_to_go}
B_{k,t}(n) = \frac{n^{k-1}}{(k-1)!}  \sum_{\delta \divides n} \frac{\mu^2(\delta)}{\delta^{k-1}} \left((-1)^{k-t+1} \binom{k-1}{t-1}\right)^{\omega(\delta)}
\sum_{\substack{j_1,\ldots,j_k\le n/\delta
\\ \gcd(j_1,\ldots,j_k)=1\\ \gcd(j_1\cdots j_k,\delta)=1 }} \frac{\psi_{k,t}(j_1,\ldots,j_k)}{j_1\cdots j_k}
\end{equation} 
\begin{equation} \label{Berror_ezaz}
+ O\left(n^{k-2} \sum_{\delta \divides n} \frac{\mu^2(\delta)}{\delta^{k-2}} \binom{k-1}{t-1}^{\omega(\delta)}
\sum_{j_1,\ldots,j_k\le n/\delta} \frac{(j_1+\cdots+j_k)|\psi_{k,t}(j_1,\ldots,j_k)|}{j_1\cdots j_k}\right).
\end{equation}

The inner sum $J$ in \eqref{B_to_go} can be written as
\begin{equation}  \label{def_J_R_tilde}
J:= \sum_{\substack{j_1,\ldots,j_k\le n/\delta \\ \gcd(j_1,\ldots,j_k)=1\\ \gcd(j_1\cdots j_k,\delta)=1 }} 
\frac{\psi_{k,t}(j_1,\ldots,j_k)}{j_1\cdots j_k}  = \sum_{\substack{j_1,\ldots,j_k =1 \\ \gcd(j_1,\ldots,j_k)=1\\ 
\gcd(j_1\cdots j_k,\delta)=1 }}^{\infty} \frac{\psi_{k,t}(j_1,\ldots,j_k)}{j_1\cdots j_k} + \widetilde{R}(n,\delta), 
\end{equation}
where for the remainder term $\widetilde{R}(n,\delta)$ one has
\begin{equation} \label{first_est_R_tilde}
\widetilde{R}(n,\delta) \ll  \sideset{}{'}\sum_{j_1,\ldots,j_k} \frac{|\psi_{k,t}(j_1,\ldots,j_k)|}{j_1\cdots j_k}
\end{equation}
with $\sum^{'}$ meaning that there exists at least one $i$ 
($1\le i\le k$) such that $j_i>n/\delta$.

Hence, according to \eqref{B_to_go}, \eqref{def_J_R_tilde} and using Lemma \ref{Lemma_V_delta} the main term for $B_{k,t}(n)$ is 
\begin{equation*}
\frac{n^{k-1}}{(k-1)!}  \prod_p \frac{G_{k,t}(p)}{p^{k-1}} \sum_{\delta \divides n} \frac{\mu^2(\delta)}{\delta^{k-1}}
\left((-1)^{k-t+1} \binom{k-1}{t-1}\right)^{\omega(\delta)}  \prod_{p\divides \delta} \left(\frac{G_{k,t}(p)}{p^{k-1}}\right)^{-1}
\end{equation*}
\begin{equation*}
= \frac{n^{k-1}}{(k-1)!}  \prod_p \frac{G_{k,t}(p)}{p^{k-1}}
\prod_{p\divides n} \left(1 + \frac{(-1)^{k-t+1}\binom{k-1}{t-1}}{G_{k,t}(p)} \right)
= \frac{n^{k-1}}{(k-1)!} D_{k,t} g_{k,t}(n).
\end{equation*}

Now we estimate the error $\widetilde{R}_1$ coming from \eqref{first_est_R_tilde}. By \eqref{B_to_go}, \eqref{def_J_R_tilde}, \eqref{first_est_R_tilde} 
and applying \eqref{estimate_2} in the case $x=n/\delta$ we obtain that the error $\widetilde{R}_1$ is as follows. If $k\ge t\ge 3$, then 
\begin{equation*}
\widetilde{R}_1 \ll n^{k-2} \sum_{\delta \divides n} \frac{\mu^2(\delta)}{\delta^{k-2}} \binom{k-1}{t-1}^{\omega(\delta)} =
n^{k-2} \prod_{p\divides n} \left(1+ \frac1{p^{k-2}} \binom{k-1}{t-1}\right),
\end{equation*}
which is $\ll n^{k-2}$ if $k\ge 4$. If $k=t=3$, then $\widetilde{R}_1 \ll n \prod_{p\divides n} \left(1+ \frac1{p} \right) \ll n\log \log n$.
However, formula \eqref{R_k_n_asympt} shows that in the latter case the final error term is $O(n)$.

If $k=3$, $t=2$, then
\begin{equation*}
\widetilde{R}_1 \ll n \sum_{\delta \divides n} \frac{\mu^2(\delta)}{\delta} 2^{\omega(\delta)} \left(\log n/\delta \right)^2 \le
n(\log n)^2 \prod_{p\divides n} \left(1+ \frac{2}{p}\right)
\end{equation*}
\begin{equation*}
\le n (\log n)^2 \prod_{p\divides n} \left(1+ \frac1{p}\right)^2
\ll n (\log n)^2 (\log \log n)^2.
\end{equation*}

To estimate the error $\widetilde{R}_2$ under \eqref{Berror_ezaz} we apply \eqref{estimate_1}. We deduce that if $t\ge 3$, then 
\begin{equation*}
\widetilde{R}_2 \ll n^{k-2} \sum_{\delta \divides n} \frac{\mu^2(\delta)}{\delta^{k-2}} \binom{k-1}{t-1}^{\omega(\delta)} =
n^{k-2} \prod_{p\divides n} \left(1+ \frac1{p^{k-2}} \binom{k-1}{t-1}\right),
\end{equation*}
which is $\ll n^{k-2}$ if $k\ge 4$ and is $ \ll n \prod_{p\divides n} \left(1+ \frac1{p} \right) \ll n\log \log n$ if $k=t=3$.
However, in the latter case the final error term is $O(n)$ by \eqref{R_k_n_asympt}, already mentioned above.

If $t=2$, $k\ge 3$, then we have
\begin{equation*}
\widetilde{R}_2 \ll n^{k-2} \sum_{\delta \divides n} \frac{\mu^2(\delta)}{\delta^{k-2}} (k-1)^{\omega(\delta)} \left( \log n/\delta \right)^{k-1}\le
n^{k-2} (\log n)^{k-1} \prod_{p\divides n} \left(1+ \frac{k-1}{p^{k-2}}\right)
\end{equation*}
which is $\ll n^{k-2}(\log n)^{k-1}$ if $k\ge 4$, and is $\ll n (\log n)^2 (\log \log n)^2$ if $k=3$.

This completes the proof of Theorem \ref{Theorem_B}.
\end{proof}

\section{Acknowledgement} The author thanks the referee for useful comments and suggestions concerning the presentation of the proofs of the paper.

\end{document}